\pdfoutput=1
\documentclass[final]{siamonline1116}
\usepackage{slashbox}
\usepackage[utf8]{inputenc}
\usepackage{bm}
\usepackage{graphicx}
\usepackage{amsmath}
\usepackage[parfill]{parskip}
\usepackage{xcolor}
\usepackage[english]{babel}
\usepackage{amsfonts}
\usepackage{amssymb}
\usepackage{amsbsy}
\usepackage{hyperref}
\usepackage{enumitem}
\usepackage{epstopdf}
\usepackage{subcaption}
\usepackage{cleveref}
\usepackage{soul}
\usepackage{tabularx}
\usepackage{mathtools}
\usepackage{placeins}
\usepackage{siunitx}
\usepackage{dsfont}
\usepackage{amsopn}

\ifpdf
  \DeclareGraphicsExtensions{.eps,.pdf,.png,.jpg}
\else
  \DeclareGraphicsExtensions{.eps}
\fi

\numberwithin{theorem}{section}

\newsiamremark{remark}{Remark}
\newsiamremark{Example}{Example}

\newcommand{\minus}{\scalebox{0.75}[1.0]{$-$}} 
\newsiamthm{assumption}{Assumption}

\newcommand{\R}{\mathcal{R}}

\newcommand{\fatQ}{\mathbb{Q}}
\newcommand{\fatR}{\mathbb{R}}

\newcommand{\fatS}{\mathbb{S}}

\newcommand{\TheTitle}{
Non-Uniform 
Stability, Detectability, and, Sliding Mode Observer Design for Time Varying Systems with Unknown Inputs
}
\newcommand{\ShortTitle}{}
\newcommand{\TheAuthors}{M. Tranninger, S. Zhuk, M. Steinberger, L. Fridman, M. Horn}

\ifpdf
\hypersetup{
  pdftitle={\TheTitle},
  pdfauthor={\TheAuthors}
}
\fi

\headers{\ShortTitle}{\TheAuthors}

\title{{\TheTitle}
	\and
}

\author{
Markus Tranninger\thanks{Institute of Automation and Control, Graz University of Technology, Austria	(\email{markus.tranninger@tugraz.at}).}
\and
Sergiy Zhuk\thanks{IBM Research, Damastown, Dublin 15, Ireland
	(\email{sergiy.zhuk@ie.ibm.com}).}
\and
Martin Steinberger\footnotemark[2]
\and
Leonid Fridman\thanks{Facultad de Ingenier\'{i}a,
	Universidad Nacional Aut\'{o}noma de M\'{e}xico, M\'{e}xico D.F., Mexico
	(\email{lfridman@unam.mx})}
\and
Martin Horn\footnotemark[2]
}

\begin{document}
\maketitle

\begin{abstract}
This paper discusses stability and robustness properties of a recently proposed observer algorithm for linear time varying systems. The observer is based on the approximation and subsequent modification of the non-negative Lyapunov exponents which yields a (non-uniform) exponentially stable error system. Theoretical insights in the construction of the observer are given and the error system is analyzed with respect to bounded unknown inputs. Therefor, new conditions for bounded input bounded state stability for linear time varying systems are presented. It is shown that for a specific class of linear time varying systems, a cascaded observer based on higher order sliding mode differentiators can be designed to achieve finite time exact reconstruction of the system states despite the unknown input. A numerical simulation example shows the applicability of the proposed approach.
\end{abstract}

\begin{keywords}
linear time varying system, Lyapunov exponents, state estimation, exact reconstruction, strong observability, sliding mode
\end{keywords}



\section{Introduction}
Observers which are robust with respect to unknown inputs, so-called unknown input observers are used in a large number of different applications like fault detection~\cite{JIE1996,Patton1989,Edwards2000} or decentralized state estimation~\cite{Pillosu2011}. Most of the contributions deal with linear time invariant (LTI) systems which fulfill the so-called rank condition~\cite{Patton1989}. In~\cite{Edwards2000}, sliding mode unknown input observers for LTI systems which fulfill the rank condition are presented. In~\cite{Ichalal2015a,Chen2017} this work is extended to linear parameter varying (LPV) systems where the system matrices depend on a scheduling parameter. It is shown in~\cite{Bejarano2007} that the rank condition can be relaxed by using higher order sliding mode estimation schemes. Recently,~\cite{Galvan-Guerra2016} generalized the ideas of~\cite{Bejarano2007} to linear time varying (LTV) systems. This generalization is worthwhile as a large number of problems in control and estimation can be formulated as LTV problems including LPV systems or the linearization of nonlinear systems around a trajectory. \par
The observer presented in~\cite{Galvan-Guerra2016} is based on a deterministic interpretation of the Kalman filter, a so-called Least Squares filter~\cite{Willems2004}. This is a particular case of the general minimax filter dealing with bounded unknown inputs and measurement noise as presented in~\cite{SZAPOGN_TAC16}. A main issue regarding this class of observers is the computational complexity. The feedback gain is given by the solution of a differential Riccati equation which might become computationally intractable for systems with a large number of states. Furthermore, geometrical numerical integration schemes must be used to preserve all the quadratic invariants of the observer error equation, e.g. the Lyapunov function which decays along the estimation error trajectory, for long estimation durations~\cite{Frank2014}. Such schemes are typically implicit and require solving the corresponding linear matrix Hamiltonian system which doubles the dimension of the state space and makes the computations even more expensive.
In~\cite{Frank2017} a remedy of this problem in the form of a  numerically efficient observer algorithm is presented. However, unknown inputs are not considered in~\cite{Frank2017}. \par
This work proposes a numerically efficient observer algorithm as an extension of~\cite{Frank2017} for systems with unknown inputs. The observer gain is computed in the sense that only unstable modes of the original system are stabilized in the observer error dynamics. This may reduce the computational effort as typically the dimension of the non-stable subspace is much lower than the dimension of the whole state space. \par
The observer is based on a directional detectability condition which is utilized to stabilize the error equation if the non-stable subspace is regularly observed as time goes towards infinity. This detectability condition is discussed in detail and compared with classical detectability notions for linear time varying systems. The advantage of the observer presented in~\cite{Frank2017} is that it is not required to solve a differential Riccati equation, but a numerically stabilized matrix valued ordinary differential equation on a reduced state space only. The aforementioned advantage comes at a price: the observer cannot be made arbitrarily fast as it does not affect modes other than the non-stable ones even if the system is observable. Furthermore, it only guarantees non-uniform exponential convergence. Thus, conditions for bounded input bounded state stability are presented in this contribution to guarantee stability of the error system in the presence of bounded disturbance inputs. Imposing additional assumptions on the considered system, it is shown that combining the ideas of~\cite{Frank2017} and higher order sliding mode techniques, it is possible to exactly reconstruct the system states in finite time despite an unknown but bounded input.  \par

The paper is structured as follows: in~\cref{sec:preliminaries}, important concepts of Lyapunov stability theory and the computational aspects of the approximation of Lyapunov exponents are summarized. Furthermore the concept of strong forward regularity is introduced in addition to existing notions of forward regularity. Based on the strong forward regularity assumption, conditions for bounded input bounded state stability are presented in~\cref{sec:bibs}. Section \ref{sec:observer} discusses the directional detectability condition and the observer algorithm presented in~\cite{Frank2017}. It is shown that under some additional assumptions, the error system of the proposed observer can be rendered bounded input bounded state stable. It is shown in~\cref{sec:finite_time_reconstruction} that in combination with a higher order sliding mode reconstruction scheme, the system states can be reconstructed in finite time despite the unknown input and the cascaded observer design is carried out step-by-step. Finally, a numerical simulation example given in~\cref{sec:example} shows the applicability of the proposed approach and~\cref{sec:conclusion} concludes the paper.

\paragraph{Notation}
If not stated otherwise, $\|\cdot\|$ denotes the 2-norm of a vector or the corresponding induced matrix norm. If $Q\in\mathds{R}^{n\times m}$ is an orthogonal basis for a subspace in $\mathds{R}^{n\times n}$, $Q_\bot \in \mathds{R}^{n \times n-m}$ is its orthogonal complement. For a Matrix $A\in\mathds{R}^{n\times n}$, $a_{ij}$ denotes the element in the $i$-th row and $j$-th column.

\section{Preliminaries}\label{sec:preliminaries}
\subsection{General Preliminaries}

\begin{definition}[Quasi integrability]
	Let $f(x)$ be a scalar function. Then, $f(x) = f^+(x) - f^-(x)$ with
	\begin{equation}
	f^+(x) \coloneqq \max(f(x),0)\text{ and }f^-(x) \coloneqq \max(-f(x),0)
	\end{equation}
	where $f^+$ and $f^-$ are non-negative. On an interval $[a,b]$ the integrals
	\begin{equation}\label{eq:integrals}
	\int_{a}^{b}f^+(x)dx \text{ and } 	\int_{a}^{b}f^-(x)dx
	\end{equation}
	are either finite or infinite. If one of the integrals in~\eqref{eq:integrals} is finite, then $f(x)$ is quasi-integrable on $[a,b]$ and
	\begin{equation}\label{eq:sumintegrals}
	\int_{a}^{b}f(x)dx \coloneqq \	\int_{a}^{b}f^+(x)dx -\int_{a}^{b}f^-(x)dx.
	\end{equation}
	If both integrals in~\eqref{eq:integrals} are finite, $f(x)$ is integrable on $[a,b]$.
\end{definition}
Note that~\eqref{eq:sumintegrals} is not defined if both integrals in~\eqref{eq:integrals} are infinite.

\subsection{Lyapunov Basis and Lyapunov Exponents}
This section recalls basic concepts of Lyapunov stability theory for linear time varying systems. Consider the system
\begin{equation}\label{eq:lyapsystem}
\dot{x}(t,t_0,x_0) = A(t)x(t,t_0,x_0),\; x\in\mathds{R}^n,\; x(t_0) = x_0
\end{equation}
with $t_0$ as initial time and $A(t)\in\mathds{R}^{n\times n}$ continuous and bounded. Define the function
\begin{equation}
\lambda (x_0) = \limsup_{t\rightarrow \infty} \frac{1}{t-t_0}\log \|x(t,t_0,x_0) \|
\end{equation}
which measures the asymptotic rate of exponential growth or decay of the solution of~\eqref{eq:lyapsystem} with initial condition $x_0$. Considering every $x_0\in\mathds{R}^n$, this function attains at most $n$ distinct values \mbox{$\lambda_1,\ldots,\lambda_s,\;s\leq n$} \cite{Barreira2002}. These numbers are called Lyapunov exponents. Moreover, $d_i\in\mathds N$ is the multiplicity of the corresponding Lyapunov exponent $\lambda_i$ such that $d_1+\cdots +d_s=n$. Without loss of generality it is assumed that the Lyapunov exponents are ordered such that $\lambda_1 \geq \lambda_2 \geq \cdots \geq \lambda_{s}$.
\begin{definition}[Ordered normal Lyapunov basis]
	Under the assumption of $n$ distinct Lyapunov exponents, a basis \mbox{$\mathcal{V} = \left[v_1\;v_2\;\cdots v_n\; \right] \;\in \mathds{R}^{n\times n}$} such that \mbox{$\lambda(v_1)= \lambda_1$}, \mbox{$\lambda(v_2)=\lambda_2$, \ldots,} $\lambda(v_n)=\lambda_n$ is called ordered normal Lyapunov basis.
\end{definition}
A more general formulation where the multiplicity of the Lyapunov exponents may be larger than one is shown in the proof of Theorem 1.2.2. in~\cite{Barreira2002}. The subspace which is spanned by \mbox{$\mathcal{V}_u = \left[v_1\;v_2\;\cdots \;v_k \right] \;\in \mathds{R}^{n\times k}$} such that the corresponding \mbox{$\lambda_1,\;\ldots,\;\lambda_k \geq 0$} is called non-stable tangent space. It is well known that system~\eqref{eq:lyapsystem} is non-uniformly exponentially stable if and only if $\lambda_1<0$ as
\begin{equation}
\forall t_0,\,\epsilon>0\; \exists C_\epsilon >0:\; \|x(t,t_0,x_0)\|\leq C_\epsilon e^{(\lambda_1+\epsilon)(t-t_0)}\|x_0\|,\; \forall t\geq t_0,
\end{equation}
holds, see~\cite[p. 9]{Barreira2002}. Thus, if $\dim \mathcal{V}_u=0$ which is equivalent to $\lambda_1<0$, for any sufficiently small $\epsilon$, $x(t) \rightarrow 0$ as $t\rightarrow \infty$.
In other words, if all Lyapunov exponents are negative, system~\eqref{eq:lyapsystem} is non-uniformly exponentially stable.\par

\subsection{Continuous QR-decomposition}\label{sec:continuousQR}
Consider the (continuous) QR decomposition of a differentiable matrix $X(t)\in\mathds{R}^{n\times m},\;n\geq m$, i.e. $X(t)=Q(t)R(t)$ where $Q\in\mathds{R}^{n\times m}$ is orthogonal i.e. $Q^TQ = I$ and $R\in\mathds{R}^{m\times m}$ is upper triangular. Differentiating $X=QR$ and $Q^TQ=I$ yields
\begin{equation}
\begin{aligned}
\dot{X}&=\dot{Q}R+Q\dot{R} \\
0 &= \dot{Q}^TQ + Q^T\dot{Q},
\end{aligned}
\end{equation}
where $Q^T\dot{Q} = S$ is a skew symmetric matrix. Details on the continuous QR decomposition can be found in~\cite{Dieci1999} and the main results are summarized in the next Lemma.
\begin{lemma}[Continuous QR decomposition,~\cite{Dieci1999}]
	Assume that $X(t)\in \mathds{R}^{n\times m},\; n\geq m$ has full column rank and is $k$ times continuously differentiable. Then, there exists a $k$ times continuously differentiable QR decomposition $X=QR$ with orthogonal $Q$ and upper triangular $R$ such that
	\begin{equation}\label{eq:continuousqr}
	\begin{aligned}
	\dot{R} &= Q^T\dot{X} - Q^T\dot{Q}R = Q^T\dot{X}-SR, \\
	\dot{Q} &= (I-QQ^T)\dot{X}R^{-1}+QS
	\end{aligned}
	\end{equation}
	with $S$ skew symmetric such that $\dot{R}$ is upper triangular. Furthermore, if $X$ is the solution of a matrix differential equation $\dot{X}(t)=A(t)X(t)$, equations~\eqref{eq:continuousqr} simplify to
	\begin{equation}\label{eq:contQRdynamical}
	\begin{aligned}
	\dot{R} &= BR \text{ with } B=(Q^TAQ-S), \\
	\dot{Q} &= (I-QQ^T)AQ+QS,
	\end{aligned}
	\end{equation}
	with $S=-S^T$ and $S_{ij}= (Q^TAQ)_{ij}$ for $i>j$.
\end{lemma}\vspace{0.1cm}
It is pointed out in~\cite[Theorem 3.1]{Dieci1999} that if $X(t)$ loses rank at some time instants and if the matrix obtained by differentiating the columns of $X$ still has full rank at these instants, the QR decomposition still exists and is unique. However, the smoothness properties might get lost.

\subsection{Computation of a Lyapunov basis and approximation of Lyapunov exponents}\label{sec:lyapexponents_computation}
The Lyapunov exponents can be approximated by using the fundamental matrix differential equation corresponding to~\eqref{eq:lyapsystem},
\begin{equation}\label{eq:fundametaleq}
\dot{X} = A(t)X(t),\;X(t)\in\mathds{R}^{n\times n},
\end{equation}
with the initial condition $X(t_0)=[x_{1,0}\;\ldots\;x_{n,0}]$ as an ordered normal Lyapunov basis~\cite{Barreira2002}. However this matrix differential equation is numerically ill conditioned. Thus, based on Perron's lemma~\cite[Lemma 1.3.3]{Barreira2002}, the authors in~\cite{Dieci1997,Frank2017} suggest to numerically approximate the Lyapunov exponents based on the transformation of~\eqref{eq:fundametaleq} to the upper triangular form
\begin{equation} \label{eq:lyaptriangular}
\dot{R} = BR
\end{equation}
by using a regular coordinate transformation \mbox{$R=Q^TX$}. This can be achieved by the full continuous time QR-decomposition~\cite{Dieci1997,Frank2017} with $B$ in the form of~\eqref{eq:contQRdynamical}. It is well known that the stability property based on Lyapunov exponents is not robust under small perturbations~\cite{Barreira2002}. However, assuming the additional condition of forward regularity, stability is preserved in the presence of small perturbations.
\begin{definition}[Forward regularity, \cite{Frank2017,Barreira2002}]
	Let \mbox{$R=[R_1,\ldots,R_n]$} be the unique solution of~\eqref{eq:lyaptriangular} and $B$ as in~\eqref{eq:contQRdynamical}. Then, the Lyapunov exponent $\lambda(R_{0,i})=\lambda_i$ is called forward regular provided
	\begin{equation}\label{eq:regularitycond}
	\limsup_{t\rightarrow \infty} \frac{1}{t-t_0}\int_{t_0}^t B_{ii}(s)ds = \liminf_{t\rightarrow \infty} \frac{1}{t-t_0}\int_{t_0}^t B_{ii}(s)ds,
	\end{equation}
	with $B_{ii}$ as the corresponding diagonal element of $B$ in~\eqref{eq:lyaptriangular} and $R_{0,i} = R_i(t_0)$.
\end{definition}
It is pointed out in~\cite{Frank2017,Barreira2002} that even if this condition seems quite strong, it typically holds. Furthermore, regularity is provided for discrete time systems by Oseledec's theorem and thus it typically holds for systems resulting from discretization in time, see~\cite{Frank2017,oseledec}. It is shown in \cite{Frank2017} that if the system is forward regular then,
\begin{equation}\small \label{eq:lyapforwardregularity}
\lambda_i = \lambda(v_i) = \lim_{t\rightarrow \infty} \frac{1}{t-t_0}\int_{t_0}^t B_{ii}(s)ds = \lim_{t\rightarrow \infty}\frac{1}{t-t_0} \int_{t_0}^t Q_i^T(s)A(s)Q_i(s)ds,
\end{equation}
holds for any ordered Lyapunov basis $\mathcal{V}=\left[v_1\;\cdots\;v_n\right]$ and $Q_i$ as the $i$-th column of $Q$ as in~\eqref{eq:contQRdynamical} for \mbox{$m=n,$} and \mbox{$i=1,\ldots,n$}.
Due to the modified Gram-Schmidt orthogonalization~\cite{Dieci1997}, the Lyapunov exponents obtained by averaging over the diagonal elements of $B$ are ordered such that $\lambda_1\geq \lambda_2 \geq \cdots \geq \lambda_n$. To approximate the corresponding Lyapunov exponents,~\eqref{eq:lyapforwardregularity} can be computed for finite time horizons. If the dimension of the non-stable state space $k$ is known a priori, it is reasonable to approximate only the $k$ largest Lyapunov exponents. Therefor it suffices to solve~\eqref{eq:contQRdynamical} on a reduced state space only such that $\dot{X}_k = A(t)X_k(t)$ and $X_k = QR$ with $Q\in\mathds{R}^{n\times k}$ and $R\in\mathds{R}^{k\times k}$. Due to the modified Gram-Schmidt orthogonalization procedure, for almost every choice of the initial condition $X(t_0)=Q(t_0)R(t_0)$, $Q$ converges to an orthogonal basis for the non-stable tangent subspace.

Now, a stronger concept than forward regularity is introduced. It is shown in the subsequent section that strong forward regularity together with negative Lyapunov exponents is sufficient for bounded input bounded output stability.

\begin{definition}[Strong forward regularity]\label{def:strongregularity}
The Lyapunov exponent $\lambda_i$ is called strongly forward regular if it is forward regular and for an arbitrarily small $\epsilon>0$,
	\begin{equation}
	b_i(t) = \begin{cases}
	B_{ii}(t) - \epsilon & \text{ for } \lambda_i \geq 0\\
	B_{ii}(t) + \epsilon  & \text{ for } \lambda_i < 0
	\end{cases}
	\end{equation}
	is quasi integrable on $[t_0,\infty)$.
\end{definition}

\section{Bounded Input Bounded State Stability}\label{sec:bibs}
 In~\cite{Barreira2010}, conditions for stability under bounded perturbations are stated in terms of admissibility for non-uniform exponential contractions. However, the sufficient conditions given in~\cite[Theorem 6]{Barreira2010} imply a non-uniform bound such that the perturbation should vanish if time tends towards infinity. In this paper, a uniform bound on the perturbation or input is assumed. Hence, stronger conditions than negative and forward regular Lyapunov exponents are needed in this case. This section presents conditions for bounded input bounded state stability of linear time varying systems
\begin{equation}\label{eq:bibs_general}
\dot{x}(t,x_0) = A(t)x(t,x_0)+D(t)w(t),\; x\in\mathds{R}^n,\; x(t_0) = x_0,
\end{equation}
with $D(t)\in \mathds{R}^{n\times m}$ bounded and $A(t)\in \mathds{R}^{n\times n}$ continuous and bounded. Furthermore, it is assumed that the input $w(t)\in \mathds{R}^m$ is bounded according $\|w(t)\|\leq \bar w$. Note that by defining \mbox{$f(t)=D(t)w(t)$}, $f(t)$ is also bounded by some constant $\|f(t)\|\leq \bar f$.

\begin{definition}[Bounded input bounded state stability]
	System~\eqref{eq:bibs_general} is called bounded input bounded state (BIBS) stable if for every bounded input and every initial condition there exists a scalar $c(\bar w,x_0)>0$ such that $\|x(t)\|\leq c$ holds for all $t\geq t_0$.
\end{definition}

\subsection{Scalar Case}
Consider the scalar system
\begin{equation}\label{eq:bibs_scalar}
\dot{x} = a(t)x + f(t) \text{ with } |f(t)| \leq \bar{f}
\end{equation}
and $a(t)$ continuous and bounded. A sufficient condition for BIBS stability of~\eqref{eq:bibs_scalar} is given in the next theorem.
\begin{theorem}\label{thm:bibs_scalar}
	System~\eqref{eq:bibs_scalar} is bounded input bounded state stable if the Lyapunov exponent $\lambda$ of the corresponding homogeneous system \mbox{$\dot{x}=a(t)x$} is strongly forward regular and $\lambda<0$.
\end{theorem}
\begin{proof}
The solution of~\eqref{eq:bibs_scalar} can be written as
\begin{equation}
x(t) = \Phi(t,t_0)x_0 + \int_{t_0}^{t}\Phi(t,\tau)f(\tau)d\tau,
\end{equation}
with
\begin{equation}
\Phi(t,t_0) = e^{\int_{t_0}^{t}a(\tau)d\tau}.
\end{equation}
For a bounded input as in~\eqref{eq:bibs_scalar},
\begin{equation}\label{eq:bibs_inequ}
\begin{aligned}
|x(t)| &\leq e^{\int_{t_0}^t a(\tau)d\tau}|x_0| + \int_{t_0}^{t} e^{\int_{s}^t a(\tau)d\tau}|f(s)|ds\\
& \leq e^{\int_{t_0}^t a(\tau)d\tau}|x_0| +\bar f \int_{t_0}^{t} e^{\int_{s}^t a(\tau)d\tau}ds,
\end{aligned}
\end{equation}
holds. The bound in the last line of inequality~\eqref{eq:bibs_inequ} is exactly the solution of the auxiliary system
\begin{equation}
\dot{z} = a(t)z + \bar f \text{ with } z_0=z(t_0) = |x_0|.
\end{equation}
This system can be recast as an homogeneous second order system of the form
\begin{equation}\label{eq:scalar_bibs_extended}
\begin{bmatrix}
\dot{z} \\
\dot{\xi}
\end{bmatrix} = \begin{bmatrix}
a(t) & \epsilon \\0 & 0
\end{bmatrix}\begin{bmatrix}
z \\ \xi
\end{bmatrix},
\end{equation}
with an arbitrarily small $\epsilon > 0$ and
\begin{equation}
\xi_0 = \xi(t_0) =  \frac{1}{\epsilon}\bar f,\;z_0 = |x_0|.
\end{equation}
Note that due to~\eqref{eq:bibs_inequ}, stability of~\eqref{eq:scalar_bibs_extended} implies BIBS stability of~\eqref{eq:bibs_scalar}. To derive the stability condition stated in \cref{thm:bibs_scalar}, consider the bound on the state transition matrix based on Lozinskii's estimate using logarithmic norms. Following the results presented in~\cite[Theorem 4.6.3]{Adrianova1995}, the transition matrix can be bounded according to
\begin{equation}\label{eq:scalar_bound}
\|\Phi(t,t_0)\|_{\infty} \leq \text{exp}\left[{\int_{t_0}^{t} \max_\mu\left( a_{\mu\mu}(s)+\sum_{\mu\neq\zeta}|a_{\mu\zeta}(s)| \right) ds}\right]
\end{equation}
with $\|.\|_\infty$ as the induced row sum norm and $a_{\mu\zeta}$ as the coefficients of the system matrix. Writing~\eqref{eq:scalar_bound} explicitly for system~\eqref{eq:scalar_bibs_extended} gives
\begin{equation}\label{eq:bound_scalar}
\|\Phi(t,t_0)\|_\infty \leq e^{\int_{t_0}^{t} \max\left[a(s)+\epsilon,0\right] ds}.
\end{equation}
As by assumption in~\cref{thm:bibs_scalar}, the Lyapunov exponent of $\dot{x} = a(t)x$ is negative,
\begin{equation}
\lambda = \lim_{t\rightarrow \infty} \frac{1}{t-t_0}\int_{t_0}^t a(s)ds < 0,
\end{equation}
holds. Thus, one can select a sufficient small $\epsilon$ such that
\begin{equation}
\lim_{t\rightarrow \infty} \frac{1}{t-t_0}\int_{t_0}^t \left[a(s)+\epsilon\right] ds = \lambda +\epsilon < 0.
\end{equation}
Together with the strong forward regularity assumption, the integral can be written as
\begin{equation}
\int_{t_0}^t \left[a(s)+\epsilon\right] ds = \int_{t_0}^t \max\left[a(s)+\epsilon,0\right]ds - \int_{t_0}^t \max\left[-a(s)-\epsilon,0\right]ds.
\end{equation}
Now as
\begin{equation}
\lim_{t\rightarrow \infty} \int_{t_0}^t \left[a(s)+\epsilon\right] ds = -\infty,
\end{equation}
by quasi-integrability the integral
\begin{equation}
\int_{t_0}^t \max(a(s)+\epsilon,0)ds
\end{equation}
converges for $t\rightarrow\infty$ and thus
\begin{equation}
\lim_{t\rightarrow \infty}e^{\int_{t_0}^{t} \max\left[a(s)+\epsilon,0\right] ds} < \infty.
\end{equation}
This completes the proof.
\end{proof}

\subsection{General case}

Now, the results of the previous section are generalized to system~\eqref{eq:bibs_general}.
Let $X=\fatQ\fatR$ be the unique full QR decomposition of the fundamental matrix $\dot{X}=A(t)X,\,X\in\mathds{R}^{n\times n}$ and transform system~\eqref{eq:bibs_general} in upper-triangular form with
\begin{equation}
\zeta \coloneqq \fatQ^T x,\; \varphi \coloneqq \fatQ^T D(t)w(t).
\end{equation}
Then,
\begin{equation}\label{eq:bibs_uppertriangular}
\dot{\zeta}=B(t)\zeta+\varphi,\,\zeta_0 =\zeta(t_0)= \fatQ^Tx_0\,\in\mathds{R}^n\,
\end{equation} is in upper triangular form with $B(t)$ as in~\eqref{eq:contQRdynamical} and  $\|\zeta(t)\| = \|x(t)\|$. Note that if the system is forward regular, the diagonal elements of $B(t)$ correspond to the Lyapunov exponents according to
\begin{equation}
\lambda_i = \lim_{t\rightarrow \infty}\frac{1}{t-t_0}\int_{t_0}^{t}B_{ii}(s)ds
\end{equation}
as stated in~\eqref{eq:lyapforwardregularity}.
\begin{theorem}\label{thm:bibs}
	If all Lyapunov exponents of~\eqref{eq:bibs_general} in the homogeneous case, i.e. $D(t)w(t)=0,\forall\; t$, are strongly forward regular and negative, system~\eqref{eq:bibs_general} is bounded input bounded state stable.
\end{theorem}
\begin{proof}
Consider the last equation $\dot{\zeta}_n = B_{nn}\zeta_n + \varphi_n$ of~\eqref{eq:bibs_uppertriangular}. Clearly, this is the scalar case as $\varphi_n$ is bounded and thus it reduces to the condition stated in theorem~\ref{thm:bibs_scalar}. Next, consider
\begin{equation}\label{eq:bibs_proof}
\dot{\zeta}_{n-1} = B_{n-1,n-1}\zeta_{n-1} + B_{n-1,n}\zeta_n + \varphi_{n-1}.
\end{equation}
If $B_{nn}$ fulfills the condition for the scalar case, $\zeta_n$ is bounded and~\eqref{eq:bibs_proof} can be recast to \mbox{$\dot{\zeta}_{n-1} = B_{n-1,n-1}\zeta_{n-1} + \tilde{\varphi}_{n-1}$} with $\tilde{\varphi}_{n-1}=B_{n-1,n}\zeta_n + \varphi_{n-1}$. The auxiliary input $\tilde{\varphi}_{n-1}$ is bounded and thus this sub-system as well is BIBS stable if the Lyapunov exponent corresponding to $B_{n-1,n-1}$ is negative fulfills the strong regularity condition. Repeating this argument for every component of $\zeta$ in~\eqref{eq:bibs_uppertriangular} completes the proof.
\end{proof}

In the work of Perron~\cite{Perron1930}, necessary and sufficient conditions for BIBS stability are given. However, these conditions might be hard to verify in practice whereas the conditions stated in~\cref{thm:bibs} in fact imply that the diagonal coefficients $B_{ii}(t)$ of the system in triangular form are non-negative only on a finite number of compact time intervals. \\

It is shown in~\cite{Kalman1960} that for linear systems in the form of~\eqref{eq:bibs_general}, bounded input bounded state stability is equivalent to uniform asymptotic stability and uniform exponential stability of the corresponding homogeneous system under the additional assumption that there exists some scalars $c_1\leq c_2$ such that
\begin{equation}\label{eq:bibs_assumption}
0<c_1\leq \|D(t)w\|\leq c_2 < \infty \text{ for all } \|w\| = 1,\; t\geq t_0.
\end{equation}
In~\cite{Silverman1968}, this assumption is relaxed to uniform complete controllability as introduced by Kalman \cite{Kalman1960_optimalcontrol}. Forward regular and negative Lyapunov exponents only guarantee (non-uniform) exponential stability. As an extension of this concept, the strong forward regularity assumption together with negative Lyapunov exponents guarantees uniform convergence of the homogeneous system under the aforementioned assumptions. Now, consider the output
\begin{equation}
y(t)=C(t)x(t)
\end{equation}
of system \eqref{eq:bibs_general} with $C(t)\in\mathds{R}^{r\times n}$ bounded. It is well known that $y(t)$ is bounded if~\eqref{eq:bibs_general} is BIBS stable~\cite{Silverman1968}. This suggests that the input-ouput operator \[
f\mapsto y: y(t)=\int_{t_0}^{t}C(t)\Phi(t,\tau)f(\tau)d\tau
\] is a bounded linear operator from $L^\infty(t_0,+\infty,\mathds{R}^m)$ to $L^\infty(t_0,+\infty,\mathds{R}^r)$.

\section{Detectability and Observer Design}\label{sec:observer}

Consider the system
\begin{equation}\label{eq:system}
\begin{aligned}
\dot{x}(t) &= A(t)x(t) + F(t)u(t) +D(t)w(t) \\
y(t) &= C(t)x(t)
\end{aligned}
\end{equation}
with the state $x\in\mathds{R}^n$, the known input $u\in\mathds{R}^q$, the unknown input $w\in\mathds{R}^m$, the system output $y\in\mathds{R}^r$ and known time varying matrices $A(t),\;F(t),\;D(t)$ and $C(t)$ of appropriate dimensions. Furthermore, it is assumed that the matrices $A(t)$, $D(t)$ and $C(t)$ are bounded and continuous and the unknown input is bounded according to some known constant $\|w(t)\| \leq \bar w$. Furthermore, it is assumed that all Lyapunov exponents of $\dot{X}=A(t)X$ are forward regular.\\
The goal is to design an observer following the ideas of~\cite{Frank2017} such that the error system is bounded input bounded state stable with respect to the unknown input.\\
It is well known for the linear time invariant case, that detectability is necessary and sufficient to design an asymptotically stable (e.g. Luenberger type) observer. However, the concept of detectability in the linear time varying setting is not unique and different definitions for detectability in the time varying setting exists in the literature~\cite{Tai1987}. It is shown that these definitions coincide for special classes of time varying systems like for example constant rank or periodic systems. Detectability according to Wonham~\cite{Wonham1968} is recalled now.

\begin{definition}[Detectability]\label{def:detectability}
	The pair $(A,C)$ is called detectable, if there exists an output feedback matrix $L$ such that the system
	\begin{equation}
	\dot{x} = \left[A(t)-L(t)C(t)\right]x
	\end{equation}
	is asymptotically stable.
\end{definition}

As this definition gives no constructive or verifiable conditions,~\cite{Frank2017} proposes a stabilizing feedback gain based on a directional detectability condition. The observer design together with a stability analysis of the error system is presented in the subsequent section.

\subsection{Observer Design}\label{sec:TSO}

Following~\cite[Corollary 4.4]{Frank2017}, the following observer is proposed for system~\eqref{eq:system} to achieve an exponentially stable error system if no unknown input is present.
\begin{equation}\label{eq:observer}
\begin{aligned}
\dot{\tilde{x}}(t) &= A(t)\tilde{x}(t) + F(t)u(t) + L(t)\left(y(t)-\tilde{y}(t)\right),\;
\end{aligned}
\end{equation}
with $\tilde{y}(t) = C(t)\tilde{x}(t)$ and the observer gain $L(t)$ determined according to
\begin{equation}\label{eq:gain}
L = pQ\tilde{Q}^T C^T,\; p>0,
\end{equation}
with
\begin{equation}\label{eq:Qtilde}
\tilde{Q}\tilde{R}= C^TCQ
\end{equation}
obtained by QR-decomposition and $Q$ as the solution to the matrix differential equation
\begin{equation}\label{eq:diffQ}
\begin{aligned}
\dot{Q}&= (I-QQ^T)AQ+QS,\;\ Q(0)=Q_0 \in\mathds{R}^{n\times k} \\
S &= -S^T,\; S_{ij} = (Q^TAQ)_{ij},\; i>j,
\end{aligned}
\end{equation}
where $k$ is the dimension of the non-stable subspace of $\dot{X}=A(t)X$. Thus, in contrast to the tools presented in~\cref{sec:lyapexponents_computation}, the continuous $QR$-decomposition with $Q$ as in~\eqref{eq:diffQ} is carried out on the non-stable subspace only. Again, due to the modified Gram-Schmidt orthogonalization procedure, $Q$ converges to a basis for the non-stable subspace~\cite{Frank2017}. For sake of simplicity, $\tilde{R}$ is determined by the discrete QR-decomposition at every time instant which can be achieved by using the modified Gram-Schmidt (mGS) algorithm. However, it would also be possible to derive a continuous QR decomposition in the form of~\eqref{eq:continuousqr} if $C^TCQ$ has full rank for almost all $t$, see~\cite{Dieci1999}. The proposed observer gain yields to the concept of directional detectability as introduced in~\cite{Frank2017} for a more general nonlinear setting.

\begin{definition}[Directional detectability]\label{def:directionaldetectability}
	The pair $(A,C)$ is called detectable in the direction of $Q_j$ if
	\begin{equation}\label{eq:detectability}
	\limsup_{t\rightarrow \infty} \frac{1}{t-t_0}\int_{t_0}^{t}\tilde{R}_{jj}(s)ds > 0,
	\end{equation}
	with $Q_j$ as the $j$-th column of $Q$. The pair $(A,C)$ is called asymptotically directionally detectable, if  for every $\lambda_j\geq 0$ with
	\begin{equation}
	\lambda_j = \lim_{t\rightarrow \infty} \frac{1}{t-t_0}\int_{t_0}^{t}B_{jj}(s)ds = \lim_{t\rightarrow \infty} \frac{1}{t-t_0}\int_{t_0}^{t}Q_{j}^T(s)A(s)Q_j(s)ds\geq 0
	\end{equation}
	the pair $(A,C)$ is detectable in the direction of the corresponding $Q_j$.
\end{definition}\vspace{0.2cm}

As already mentioned in~\cref{sec:continuousQR}, differentiability of the feedback gain might get lost if $C^TCQ$ is not of full rank for all $t$. These findings are summarized in the next proposition.
\begin{proposition}[Differentiability of the feedback gain]\ \label{le:differentiability} \\
	Assume that $A(t)$ and $C(t)$ are $\nu$-times continuously differentiable. Then
	\begin{equation}
	\tilde{A}(t) = A(t)-L(t)C(t)
	\end{equation}
	is also $\nu$ times continuously differentiable if $C^TCQ$ according to~\eqref{eq:Qtilde} has full rank for all $t$.
\end{proposition}
\begin{proof}
	If $A$ is $\nu$-times continuously differentiable then the same holds for $Q$ which follows from the continuous QR decomposition~\eqref{eq:contQRdynamical}. Hence, differentiability of the resulting feedback gain might only get lost if $C^TCQ$ looses rank for some time instants, see~\cite{Dieci1999}.
\end{proof}

In the next section, stability of the resulting error system is investigated.

\subsection{Stability analysis of the error system}\ \\
The dynamics of the estimation error \mbox{$e(t) = x(t)-\tilde{x}(t)$} in the presence of unknown inputs can be written as
\begin{equation} \label{eq:errordynamics}
\dot{e}(t) = \left[A(t)-L(t)C(t)\right]e(t) + D(t)w(t)
\end{equation}
with the output error defined as $e_y(t) = Ce(t)$. The existence of a stabilizing feedback gain $L$ is summarized in the next theorem.
\begin{theorem}[Stability of the observer error system]\ \\ \label{thm:stability error system}
	Assume that $w(t)=0$. Then, there exists a parameter $p>0$ for~\eqref{eq:gain} such that~\eqref{eq:errordynamics} is exponentially stable if and only if $(A,C)$ is detectable in any direction $Q_j$ of the non-stable tangent subspace.\\
	Furthermore, assume that $C^TCQ$ has full rank for almost all $t\geq t_0$ and all Lyapunov exponents corresponding to the non-detectable directions are strongly forward regular and negative. Then there exists a large enough $p>0$ such that the perturbed error system
	\begin{equation*}
	\dot{e}(t) = \left[A(t)-L(t)C(t)\right]e(t) + D(t)w(t)
	\end{equation*}
	is bounded input bounded state stable.
\end{theorem}
\begin{proof}
	The first part of the proof is stated in detail in~\cite{Frank2017} and is sketched here to understand the basic idea of the observer algorithm. Denote
	\begin{equation}
	X = \fatQ\fatR,\, X\in\mathds{R}^{n\times n}
	\end{equation}
	as the full QR-decomposition of the fundamental solution of $\dot{x}=A(t)x$, whereas
	\begin{equation}
	X_k = QR,\, X\in\mathds{R}^{n\times k}
	\end{equation}
	is the thin $QR$-decomposition with $k$ as the dimension of the non-stable tangent subspace. Using the Lyapunov transformation $z=\fatQ^T e$, the error system~\eqref{eq:errordynamics} in the new coordinates is
	\begin{equation}
	\begin{aligned}\label{eq:errortransformed}
	\dot{z} &= \left(\fatQ^T A \fatQ -\fatS  -\fatQ^T LC\fatQ\right)z +\fatQ^T Dw\\
	&= \left(B-\fatQ^T LC\fatQ\right)z + \fatQ^T Dw,
	\end{aligned}
	\end{equation}
	with $\fatS$ skew symmetric and $B$ upper triangular as pointed out in~\cref{sec:continuousQR}. Note that $\dot{\fatQ} = \fatQ\fatS$ as $\fatQ$ is square and orthogonal in this case. Furthermore, by~\eqref{eq:gain} it follows that $LC=pQ\tilde{Q}^TC^TC$ and thus
	\begin{equation}\label{eq:errorLuppertrinangular}
	\fatQ^TLC\fatQ = \begin{bmatrix}
	Q^T\\
	Q_\bot^T
	\end{bmatrix}pQ\tilde{Q}^TC^TC\begin{bmatrix}
	Q & Q_\bot
	\end{bmatrix}=p \begin{bmatrix}
	\tilde R & \tilde Q C^T C Q_\bot \\
	0 & 0
	\end{bmatrix}
	\end{equation}
	with $\tilde{Q}$ and $\tilde{R}$ as in~\eqref{eq:Qtilde}. Since the diagonal elements $B_{ii}$ of $B$ correspond to the Lyapunov exponents of the system $\dot{x} = A(t)x$ according to
	\begin{equation}
	\lambda_i = \lim_{t\rightarrow \infty}\frac{1}{t-t_0}\int_{t_0}^{t}B_{ii}(s)\,ds,
	\end{equation}
	the Lyapunov exponents of the homogeneous error system $\dot{e}=(A-LC)e$ are
	\begin{equation}\label{eq:lyapexponents_errorsystem}
	\begin{aligned}
	\mu_i &= \lim_{t\rightarrow\infty}\frac{1}{t-t_0}\int_{t_0}^{t}Q_i^T(A-LC)Q_i\, ds \\
	&= \lambda_i -p\lim_{t\rightarrow \infty}\frac{1}{t-t_0}\int_{t_0}^{t} \tilde{R}_{ii}(s)\,ds,\,i=1,\ldots,k.
	\end{aligned}
	\end{equation}
	Since the diagonal elements $\tilde{R}_{ii}\geq 0$, it follows that the leading $k$ Lyapunov exponents can be made negative with this choice of feedback gain if and only if the directional detectability condition~\eqref{eq:detectability} of~\cref{def:directionaldetectability} holds as in this case there exists a $p>0$ such that
	\begin{equation}
	p\lim_{t\rightarrow \infty}\frac{1}{t-t_0}\int_{t_0}^{t} \tilde{R}_{ii}(s)ds>\lambda_i >0,\,i=1,\ldots,k.
	\end{equation}
	Thus, for a large enough $p>0$, the unperturbed error system $\dot{e}=(A-LC)e$ is exponentially stable.\\
	For the proof of the second part, again consider the transformed error system~\eqref{eq:errortransformed}. The matrix
	\begin{equation}
	\tilde{B} = B-\fatQ^T LC\fatQ
	\end{equation}
	is upper triangular due to~\eqref{eq:errorLuppertrinangular} and the first $k$ diagonal elements are
	\begin{equation}
	\tilde{B}_{ii} = B_{ii}-p\tilde{R}_{ii},\;i=1,\ldots,k.
	\end{equation}
	As the diagonal elements $B_{ii}$ are bounded, the corresponding $\tilde{B}_{ii}$ can be made negative for almost all $t$ if $C^TCQ$ has full rank for almost all $t$. Thus the corresponding Lyapunov exponents $\mu_i$ from equation~\eqref{eq:lyapexponents_errorsystem} are negative and strongly forward regular according to~\cref{def:strongregularity}. As the remaining Lyapunov exponents are assumed to be negative and strongly forward regular, the error system
	\begin{equation}
	\dot{e}=(A-LC)e+Dw
	\end{equation}
	is bounded input bounded state stable according to~\cref{thm:bibs} for any bounded $w(t)$.
\end{proof}

It is clear that the diagonal of $\tilde{R}$ has zero elements if $C^TCQ$ is rank deficient. Hence, from~\eqref{eq:lyapexponents_errorsystem} it follows that for every $\mu_i$, the Lyapunov exponent of the error equation which correspond to \mbox{$\lim_{t\rightarrow \infty}\frac{1}{t}\int_0^t\tilde{R}_{ii}ds = 0$} is not modified by the gain $L$. Since the maximal rank of $C^TCQ$ is $r$ it is possible to assign an arbitrary decay rate to at most $r$ Lyapunov exponents. On the other hand, if $C^T(t)C(t)Q_i(t)=0$ for some specific time instants $t$ and $Q_i(t)$ from the non-stable tangent subspace, the directional detectability condition still holds as the non-stable tangent space spanned by $Q$ should have non-trivial intersection with $\ker C^TC$ most of the time. \par
It should be pointed out that the directional detectability condition ensures the existence of a feedback gain such that $\dot{e}=(A-LC)e$ is exponentially stable. Thus, the directional detectability condition is sufficient for the detectability notion by Wonham presented in~\cref{def:detectability}. The opposite is not true which is demonstrated by a simple counter example: Consider the (linear time invariant) double integrator system
\begin{equation}
\begin{aligned}
\dot{x} &= \begin{bmatrix}
0 & 1 \\
0 & 0
\end{bmatrix} x \\
y &= \begin{bmatrix} 1 & 0\end{bmatrix}x.
\end{aligned}
\end{equation}
Clearly, this system is observable as its observability matrix
\begin{equation}
\begin{bmatrix} C \\ CA\end{bmatrix} = \begin{bmatrix} 1 & 0 \\ 0 & 1\end{bmatrix}
\end{equation}
has full rank. Thus, the system is detectable in the sense of~\cref{def:detectability}. As for linear time invariant systems, the real parts of the eigenvalues correspond to the Lyapunov exponents, there are two non-negative Lyapunov exponents $\lambda_1 = \lambda_2 = 0$. However, as $\text{rank }C=1$, the asymptotic directional detectability condition in definition~\ref{def:directionaldetectability} cannot be fulfilled in this case. Nevertheless, the detectability definition by Wonham as stated in~\cref{def:detectability} is not constructive and hard to verify in practice. Based on the directional detectability notion of~\cref{def:directionaldetectability} it is possible to design a stabilizing feedback gain. Directional detectability is discussed in detail in~\cite{Frank2017} and it is shown that it can also be used to design observer algorithms for non-linear systems.
\par
Imposing additional restrictions on the class of systems, it is possible to exactly reconstruct the system states despite the unknown input by using the presented observer and higher order sliding mode concepts. This strategy is discussed in detail in the following section.

\section{Finite Time Exact Reconstruction}\label{sec:finite_time_reconstruction}
For a specific class of linear time varying systems it is possible to exactly reconstruct the states despite the unknown input.  The concept of strong observability and state reconstruction based on the output and its derivatives are recalled. The proposed exact reconstruction method is a cascaded observer structure based on the observer algorithm presented in~\cref{sec:observer} and a higher order sliding mode (HOSM) corrector. The output error of the observer together with its derivatives obtained by Levant's robust exact differentiator is then used to correct the ``wrong'' observer estimates.

\subsection{Strong observability of linear time varying systems}\label{sec:observability}
As the known input $u(t)$ can always be canceled out in the observer error dynamics only the triple $(A(t),D(t),C(t))$ of system~\eqref{eq:system} is considered subsequently. Furthermore, it is assumed that these matrices are sufficiently differentiable and the unknown input $w(t)$ is assumed to be differentiable and bounded according to $\|w(t)\|\leq \bar w$ with a known $\bar w$. First, recall the definition of strong observability:
\begin{definition}[Strong observability~\cite{Kratz1998,Hautus1983}]\ \\ 
	The triple $(A(t),D(t),C(t))$ is called strongly observable if \mbox{$\dot{x}=A(t)x+D(t)w(t),\,C(t)x(t)= 0$} on some non-degenerate time interval implies that $x(t)= 0$ on this interval for any input $w(t)$.
\end{definition}
The generalized controllability and observability matrices for $(A(t),D(t),C(t))$ are defined as~\cite{Silverman1971}
\begin{equation}
\begin{aligned}\label{eq:observability_matrices}
Q_k(t) &= \begin{bmatrix}
P_0(t) & P_1(t) & \cdots & P_{k-1}(t)
\end{bmatrix} \in \mathds{R}^{n\times km}, \\
R_k^T(t) &= \begin{bmatrix}C_0^T(t) & C_1^T(t) & \cdots & C_{k-1}^T(t)
\end{bmatrix}\in\mathds{R}^{n\times kr}.
\end{aligned}
\end{equation}
Matrices $P_i(t)$ and $C_i(t)$, \mbox{$i =0,\,\ldots,\,k$} are recursively defined according to
\begin{equation}
\begin{aligned}
P_{i+1}(t) &= A(t)P_i(t) + \dot{P}_i(t),\;\;P_0(t) = D(t),\\
C_{i+1}(t) &= C_i(t)A(t) + \dot{C}_i(t),\;\;C_0(t) = C(t).
\end{aligned}
\end{equation}
These generalized controllability and observability matrices are used to specify the class of constant rank systems.
\begin{definition}[Constant rank system,~\cite{Silverman1969}]\ \\
	The system $(A(t),D(t),C(t))$ is called a constant rank system if there exist positive integers $\mu$, $\nu$, $q_c$ and $q_0$ such that $D(t)$, $C(t)$, and $A(t)$ are $\mu$, $\nu$ and max$(\mu,\nu)-1$ times continuously differentiable, respectively, such that
	\begin{equation}\label{eq:obsindex}
	\begin{aligned}
	\text{rank}\,Q_\mu(t) = \text{rank}\,Q_{\mu+1}(t) = q_c\leq n,\;\forall t,\\
	\text{rank}\,R_\nu(t) = \text{rank}\,R_{\nu+1}(t) = q_0\leq n,\;\forall t,\\
	\end{aligned}
	\end{equation}
	holds. The smallest integers $\mu$ and $\nu$ for which relation~\eqref{eq:obsindex} holds are called the controllability and observability index, respectively.
\end{definition}
Note that in the linear time varying case, integers $\mu,\,\nu$ may be strictly greater than $n$. To design the observer in the present contribution, conditions for strong observability together with a reconstruction method presented in~\cite{Kratz1998} are utilized. Therefor, consider the auxiliary matrices

\begin{equation}
\begin{aligned}[c]
\mathcal{D}_{\alpha+1,\alpha}(t) &= C_0(t)D(t) \\
\mathcal{D}_{\alpha+1,0}(t) &= C_{\alpha+1}(t)D(t) + \dot{\mathcal{D}}_{\alpha,0}(t)\\
\mathcal{D}_{\alpha+1,\beta}(t) &= \mathcal{D}_{\alpha,\beta-1}(t)+\dot{\mathcal{D}}_{\alpha,\beta}(t)
\end{aligned}\qquad
\begin{aligned}
\text{for} & \;0\leq\alpha\leq \nu-1, \\
\text{for} & \;1\leq\alpha\leq \nu-1, \\
\text{for} & \;1\leq\beta<\alpha\leq \nu-1. \\
\end{aligned}
\end{equation}
Based on these matrices, a necessary and sufficient condition for strong observability presented in~\cite{Kratz1998} is summarized in the next theorem.
\begin{theorem}[Strong observability,~\cite{Kratz1998}]\ \label{thm:strongobservability} \\
	The constant rank system $(A(t),D(t),C(t))$ is strongly observable on a time interval if and only if
	\begin{equation}
	\text{rank}\,S(t) = \text{rank}\,S^*(t)
	\end{equation}
	holds on this interval for
	\begin{equation}\label{eq:strongobservability} \small
	S(t) = \begin{bmatrix}
	R_\nu(t)& J_\nu(t)
	\end{bmatrix},\;\;S^*(t) = \begin{bmatrix}
	I_n & 0 \\R_\nu(t) & J_\nu(t)
	\end{bmatrix}
	\end{equation}
	with
	\begin{equation}{\label{eq:reldegreematrix}
		J_\nu (t) = \begin{bmatrix}
		0 & 0 & \cdots & 0 \\
		\mathcal{D}_{1,0} & 0 & \cdots & 0 \\
		\mathcal{D}_{2,0} & \mathcal{D}_{2,1} & \cdots & 0 \\
		\vdots & \vdots & \ddots & \vdots \\
		\mathcal{D}_{\nu-1,0} & \mathcal{D}_{\nu-1,1} & \cdots & \mathcal{D}_{\nu-1,\nu-2}
		\end{bmatrix}}\in\mathds{R}^{r\nu \times m(\nu-1)},
	\end{equation}
	and $\nu$ as the observability index.
\end{theorem}\par
\cref{thm:strongobservability} can be directly used to reconstruct the states by using the system output and its derivatives as summarized in the next proposition.
\begin{proposition}[State reconstruction,~\cite{Kratz1998}]\ \label{prop:reconstruction} \\
	Assume that the triple $(A(t),D(t),C(t))$ is a strongly observable constant rank system and $u(t)\equiv 0$. Define a matrix $K(t)\in\mathds{R}^{r\nu\times r\nu}$ such that
	\begin{equation}
	\text{Ker}\,K(t) = \text{Im}\,J_\nu(t)\;\forall t
	\end{equation}
	and furthermore let
	\begin{equation}
	H(t) = R_\nu^T(t)K^T(t)K(t)R_\nu(t).
	\end{equation}
	Then, $H(t)$ is invertible and
	\begin{equation}\label{eq:reconstruction}
	x(t) = H^{-1}(t)R_\nu^T(t) K^T(t) K(t) \hat{y}(t)
	\end{equation}
	holds, with
	\begin{equation}
	\hat{y}(t) =\begin{bmatrix}
	y^T(t)& \dot{y}^T(t) & \cdots & \left(y^{(\nu-1)}\right)^T(t)
	\end{bmatrix}^T.
	\end{equation}
\end{proposition}\vspace{0.2cm}
Note that by the choice of $K(t)$, the matrix $K(t)R_\nu(t)$ has full column rank if and only if the triple $(A(t),D(t),C(t))$ is strongly observable, see~\cite{Kratz1998} for a detailed description.

\subsection{Cascaded Observer Design}
If an additional known input $u(t)$ exists, the reconstruction concept presented in~\cref{prop:reconstruction} usually cannot be applied directly as the system would have to be strongly observable with respect to known and unknown inputs. Moreover, if system~\eqref{eq:system} is unstable, Levant's sliding mode differentiator~\cite{Levant2003} cannot be applied directly to obtain the derivatives of the output. The problem thereby is that this differentiator requires a bounded $\nu$-th derivative which does not hold for general unstable systems and thus only semi-global convergence can be guaranteed~\cite{Bejarano2007}.
Under the stated assumption of strong forward regularity, the observer presented in~\cref{sec:observer} can be designed such that it is bounded input bounded state stable. A higher order sliding mode differentiator together with the reconstruction scheme of~\cref{prop:reconstruction} is then applied to the output error $e_y$ of the observer to reconstruct the estimation error $e(t)$. Thus, the ``wrong'' state estimates $\tilde{x}$ can be corrected in this case. The cascaded observer structure is depicted in~\cref{fig:cascadestructure}. This type of cascaded observer was first introduced by~\cite{Bejarano2007} and~\cite{Fridman2007} for the linear time invariant case.

\begin{figure}
	\centering
	\includegraphics[width=0.4\linewidth]{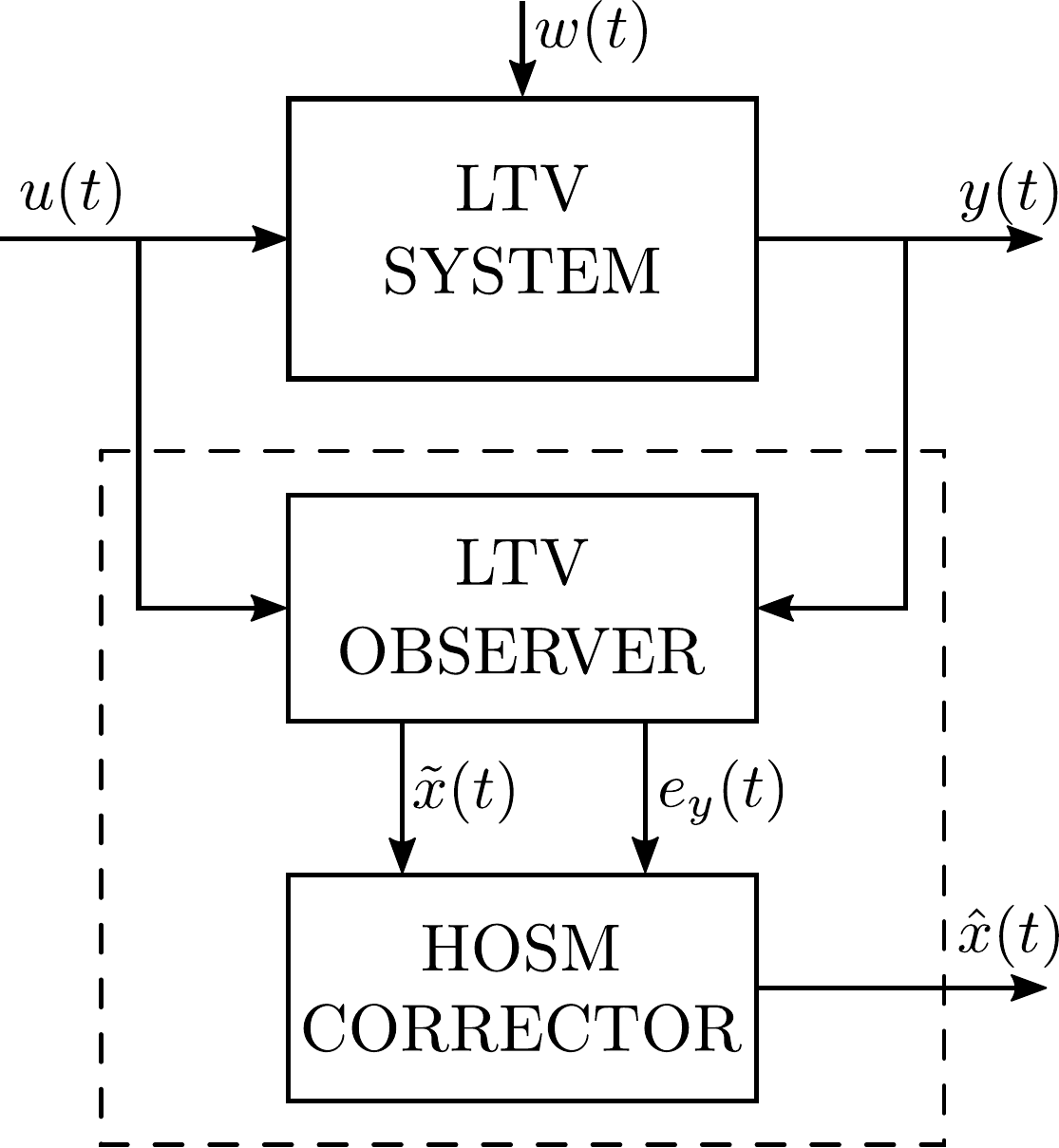}
	\caption{Cascaded observer structure.}
	\label{fig:cascadestructure}
\end{figure}

The question is under which conditions it is possible to reconstruct $e(t)$ by using $e_y(t)$ and its derivatives. Indeed, the error system should be strongly observable with respect to the unknown input which is satisfied if and only if the original system is strongly observable.
\begin{proposition}[Strong observability of the error system]\ \label{le:strongobsverror}
	The triple $(A(t)-L(t)C(t),D(t),C(t))$ is strongly observable if and only if the triple $(A(t),D(t),C(t))$ is strongly observable.
\end{proposition}
\begin{proof}
	The assumption $e_y=0$ yields \mbox{$\dot{e}(t) = A(t)e(t) + D(t)w(t)$} for~\eqref{eq:errordynamics}. This implies that $e(t)\equiv 0$ if the triple $(A(t),D(t),C(t))$ is strongly observable which proofs sufficiency. For necessity, assume that $((A(t)-L(t)C(t),D(t),C(t))$ is not strongly observable which means that one can find an input $w(t)$ such that $e_y(t)=0$ and $e(t)\neq 0$. Thus it would be also possible to find $w(t)$ such that for system~\eqref{eq:system} $y(t)=u(t) = 0$ and $x(t)\neq 0$ holds, which is a contradiction.
\end{proof}
This guarantees that if the original system is strongly observable the estimation error can be reconstructed by using the output error and its derivatives following the ideas presented in~\cref{sec:observability}, see~\cite{Kratz1998}.

\subsection{Higher Order Sliding Mode Corrector}\label{sec:HOSMcorrector}
To determine the derivatives of the output error, Levant's robust exact differentiator \cite{Levant2003} is utilized in the cascaded observer and is thus briefly summarized here. Consider an unknown smooth signal $f_0(t)$ where the $r$-th derivative $f_0^{(r)}(t)$ exists and has a known Lipschitz constant i.e. $|f_0^{(r+1)}(t)|<L$. Then, the differentiator~\cite{Levant2003} is defined in the recursive form
\begin{equation}\small\label{eq:diff_levant}
\begin{aligned}
\dot{z}_0 &= v_0  = -\lambda_r L^{1/(r-1)}|z_0-f_0(t)|^{r/(r+1)}\text{sign}(z_0-f(t))+z_1,\\
\dot{z}_1 &= v_1 =  -\lambda_{r-1} L^{1/r}|z_1-v_0|^{(r-1)/r}\text{sign}(z_1-v_0)+z_2,\\
&\vdots \\
\dot{z}_r &= -\lambda_0 L \text{sign}(z_r-v_{r-1}),
\end{aligned}
\end{equation}
with sufficiently large parameters $\lambda_i$. An established choice of parameters for a differentiator up to order 5 is $\lambda_0=1.1$, $\lambda_1=1.5$, $\lambda_2=2$, $\lambda_3=3$, $\lambda_4=5$, $\lambda_5=8$, as proposed in~\cite{Levant2003}. It is shown in~\cite{Levant2003} that in the absence of noise the equations
\begin{equation}
|z_i-f_0^{(i)}(t)| = 0, \; i=0,\,\ldots, r
\end{equation}
hold after a finite transient time and thus this differentiator can be used to exactly reconstruct the derivatives of $f_0(t)$.\\ 
This differentiator allows to reconstruct the observer error $e(t)$ by applying it component wise to the output error $e_y(t)=C(t)e(t)$. The reconstruction can be carried out by applying~\cref{prop:reconstruction} to the error system, which yields
	\begin{equation} \label{eq:reconstructiondiff}
	\tilde{e}(t) = H_e^{-1}(t)R_{\nu,e}^T(t)K_e^T(t)K_e(t)\hat{e}_y(t).
	\end{equation}
	Here, $\tilde{e}(t)$ is the estimate for $e(t)$, $H_e(t)$ and $K_e(t)$ are designed for the error system according to~\cref{prop:reconstruction} and $R_{\nu,e}$ is the observability matrix for $(A(t)-L(t)C(t),C(t))$. The output error and its derivatives are combined in
	\begin{equation}\label{eq:errorderivatives}
	\hat{e}_y = \begin{bmatrix}
	e_y^T & \dot{e}_y^T & \cdots & \left(e_y^{(\nu-1)}\right)^T
	\end{bmatrix}^T.
	\end{equation}
	If the smoothness requirements stated in\cref{le:differentiability} together with the strong regularity assumptions in~\cref{thm:stability error system} are fulfilled, the derivatives are bounded. Thus, the HOSM differentiator~\eqref{eq:diff_levant} can be applied component wise to $e_y$ which yields
	\begin{equation}
	\tilde{e}_y = \begin{bmatrix}
	e_y^T & \dot{\tilde{e}}_y^T & \cdots & \left(\tilde{e}_y^{(\nu-1)}\right)^T
	\end{bmatrix}^T,
	\end{equation}
	where $\dot{\tilde{e}}_y,\;\ldots,\;\tilde{e}_y^{(\nu-1)}$ are the estimated derivatives.
	For sufficiently large differentiator gains,~\eqref{eq:errorderivatives} can be reconstructed exactly and
	\begin{equation}
	\hat{e}_y(t) = \tilde{e}_y(t)
	\end{equation}
	holds after a finite transient time $t_f$~\cite{Levant2003}. Hence, it is possible to replace $\hat{e}_y$ with $\tilde{e}_y$ in~\eqref{eq:reconstructiondiff}. This convergence can be made theoretically arbitrarily fast by growing the differentiator gains~\cite{Levant2003}.

\subsection{Sliding mode tangent space observer}\label{sec:cascaded observer}
This section combines the previous results in order to construct the cascaded observer using a step by step design procedure. Let system~\eqref{eq:system} be strongly observable with respect to a bounded input $w(t)$ and constant rank with observability index $\nu$. To design the cascaded observer the following steps have to be carried out
\begin{enumerate}
	\item [i)] Determine the dimension of the non-stable subspace $k$ by approximating the Lyapunov exponents utilizing the continuous QR-decomposition, see~\cref{sec:continuousQR}. The number of non-negative Lyapunov exponents equals $k$.
	\item [ii)] Check if the directional detectability condition holds for the non-stable tangent subspace and if the smoothness and strong regularity assumption of~\cref{le:differentiability} and~\cref{thm:stability error system} hold.
	\item [iii)] Verify the strong forward regularity requirement according to~\cref{thm:stability error system} and choose the observer parameter $p>0$ such that this condition holds for the non-stable subspace.
	\item [iv)] Determine the matrices $R_{\nu,e}(t)$, $K_e(t)$ and $H_e(t)$ for the error system using~\cref{prop:reconstruction}.
	\item [v)] Compute the $\nu-1$ derivatives of the output error $e_y$ by component-wise application of the HOSM differentiator~\eqref{eq:diff_levant}. Use this derivatives to reconstruct the estimation errors $\tilde{e}$, see~\eqref{eq:reconstructiondiff}.
	\item [vi)] Correct the observer estimates $\tilde{x}$ according to
	\begin{equation}
	\hat{x}(t)=\tilde{x}(t)+\tilde{e}(t).
	\end{equation}
	As the sliding mode differentiator converges in finite time for sufficiently large differentiator gains,
	\begin{equation}
	x(t)=\hat{x}(t) \;\; \forall\, t\geq t_f
		\end{equation}
	 where $t_f$ represents a finite time instant.
\end{enumerate}

\section{Numerical Examples}\label{sec:example}

Consider a system in the form of~\eqref{eq:system} with the time varying dynamic matrix as
\newlength{\oldarraycolsep} 
\setlength{\oldarraycolsep}{\arraycolsep}

\setlength{\arraycolsep}{0.3em}

\begin{equation}
A(t) = \begin{bmatrix*}[r]
\minus 1.3  &  a_{12}  &  0  &   0 &  a_{15}  &  \minus0.12 &  0.42   &  0.92\phantom{1} \\
\minus a_{12} & \minus3.2  &  a_{23} &  0.51 &  a_{25} & \minus 0.23  &    0  & \minus 0.31\phantom{1}\\
0  &  \minus a_{23} &   \minus 4.4 &   0.48 &  \minus 0.80   &      0.53 &    0  &   0.17\phantom{1} \\
0  &   \minus0.51  &   \minus 0.48 &   3.35  &  0.64   &      0.59  &   0   &      a_{48} \phantom{1}\\
\minus a_{15} &    \minus a_{25} &  0.80  & \minus 0.64 &  1.80  & \minus 0.62 &   0.31 &   0.50\phantom{1} \\
0.12  &  0.23  &   \minus 0.53 &  -0.59  &   0.62  & \minus 2.45  &  \minus 0.67  & \minus 0.48\phantom{1} \\
\minus 0.42 &  0  &      0   &      0    &     \minus 0.31 &    0.67  &     \minus 3.47  &  0\phantom{1} \\
\minus 0.92 &   0.31  & \minus 0.17 &   \minus a_{48} & \minus 0.50   & 0.48 &    0 &   \minus 4.71\phantom{1}
\end{bmatrix*}
\end{equation}
with
\begin{equation*}
\begin{array}{lll}
a_{12} = 0.23\sin(0.5t), & a_{15} = \minus 0.25\sin(0.5t),   \\
a_{23} = 0.083\sin(0.3t), & a_{25} = 0.09\sin(0.3t), \\
a_{48} = \minus 0.055\sin(0.3t).
\end{array}
\end{equation*}
Moreover,
\begin{equation}
\begin{aligned}
B^T &= \begin{bmatrix*}
0 & 0 & 0 & 0 & 0 & 0 & 0.67 & 0.8 \\
0.98 & 0.63 & 0 & 0.54 & 0.54 & 0 & 0 & 0
\end{bmatrix*},\\
D^T &= \begin{bmatrix*}
0 & 0 & 0 & 0 & 0.89 & 1.4 & 0 & 0
\end{bmatrix*},
\end{aligned}
\end{equation}
and
\begin{equation}
C =  \begin{bmatrix}
1 & 0 & 0 & 0& 0& 0& 0& 0\\
0& 1& 0& 0& 0& 0& 0& 0 \\
0& 0& 1& 0& 0& 0& 0& 0 \\
0& 0& 0& 1& 0& 0& 0& 0
\end{bmatrix}.
\end{equation}
The unknown input is specified as
\begin{equation}
w(t) = 0.3+10\sin(2\pi 0.1 t) + 3\sin(2\pi 0.4t).
\end{equation}
As the system is unstable the simulations are carried out for the stabilized system. However, as $u(t)$ is known, the controller design does not influence the observer error dynamics. The integration of~\eqref{eq:diffQ} is carried out using a so-called projected RK4 which is a standard fourth order Runge-Kutta (RK4) in combination with a modified Gram-Schmidt orthogonalization to keep $Q$ orthogonal~\cite{Dieci1997}. The remaining part of the simulation is carried out using RK4 integration. The step size is \SI{1}{ms}. The sliding mode differentiator is implemented in discrete time using the toolbox presented in~\cite{Reichhartinger2017} to avoid discretization chattering. \par
It can be verified that $(A(t),D(t),C(t))$ is a constant rank system with observability index $\nu=2$. The non-stable tangent space is of dimension $k=2$, which was determined by approximating the largest three Lyapunov exponents over a simulation duration of \SI{200}{s}. The results together with the determined asymptotes are depicted in~\cref{fig:lyapexponents}. It is sufficient to solve~\eqref{eq:diffQ} for the reduced dimension $k=2$ and the tuning parameter of the observer is selected as $p=30$. The system states are depicted in~\cref{fig:states} and the norm of the reconstruction error \mbox{$\|x(t)-\hat{x}(t)\|=\sqrt{(x-\hat{x})^T(x-\hat{x})}$} is shown in~\cref{fig:errornorm}.

The error of the observer without correction is bounded despite the unknown input. Due to the numerical implementation the error of the cascaded observer (Obsv. + HOSM) converges to a small vicinity of the origin. Note that due to the chosen discretization scheme for the HOSM differentiator, discretization chattering is avoided successfully.
\begin{figure}
	\centering
	\includegraphics[width=0.6\linewidth]{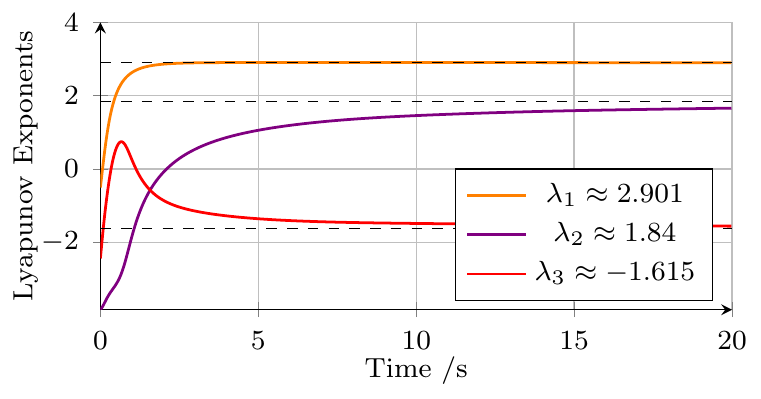}
	\caption{Approximated Lyapunov exponents together with the corresponding asymptotes.}
	\label{fig:lyapexponents}
\end{figure}

\begin{figure}
	\centering
	\includegraphics[width=0.9\linewidth]{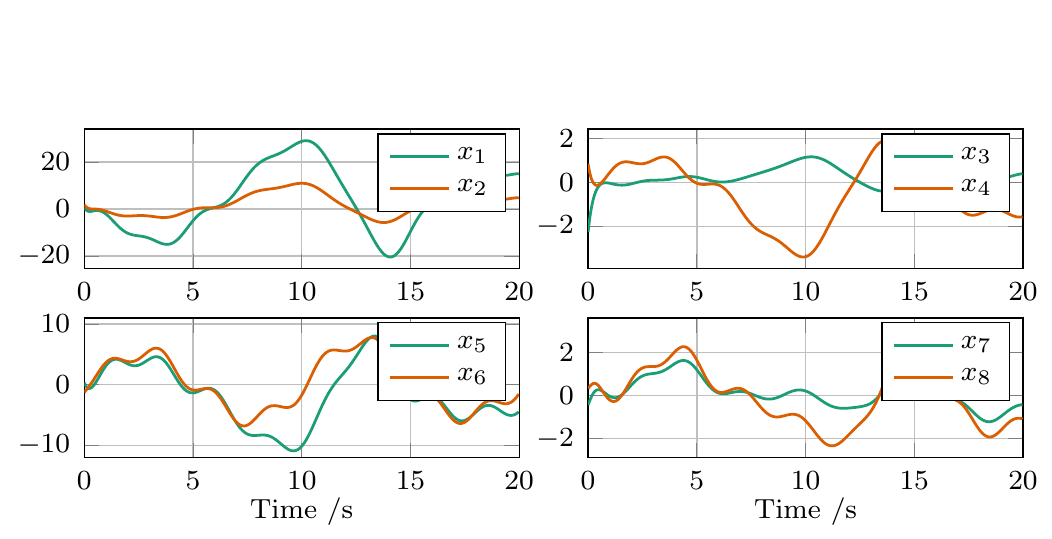}
	\caption{System states.}
	\label{fig:states}
\end{figure}

\begin{figure}
	\centering
	\includegraphics[width=0.7\linewidth]{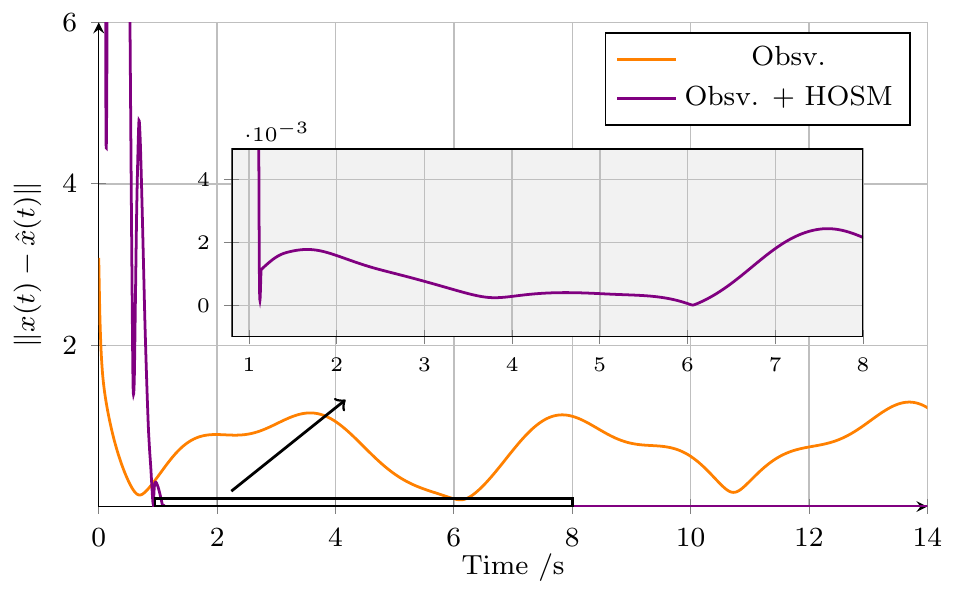}
	\caption{Norm of the estimation error for TSO and and TSO with HOSM corrector.}
	\label{fig:errornorm}
\end{figure}

The logarithmic estimation errors are shown in~\cref{fig:logerror}. It is evident that the estimation errors for the measured states $x_1$ to $x_4$ are very small as expected. As the states $x_5$ to $x_8$ are not directly measured, the corresponding estimation errors are larger but still have an order of magnitude of approximately $10^{-3}$ and less.
\begin{figure}
	\centering
	\includegraphics[width=0.7\linewidth]{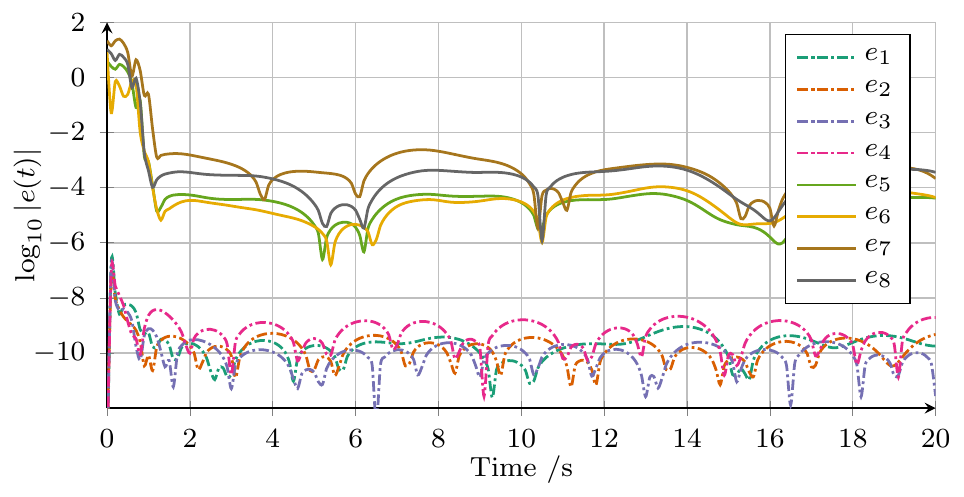}
	\caption{Logarithmic estimation error for all states.}
	\label{fig:logerror}
\end{figure}

Finally, a simulation is carried out adding Gaussian measurement noise with standard deviation \mbox{$\sigma = 10^{-3}$}. For the reconstruction, a filtered version of $e_y$ is used in $\hat{e}_y$ of~\eqref{eq:reconstructiondiff} which is inherently generated by the HOSM differentiator. This is done to mitigate the measurement noise in the reconstruction. The result for the estimates of $x_7$ and $x_8$, which show the largest errors in the noise free case, are depicted in~\cref{fig:noisystates}. It can be concluded that the proposed method achieves satisfying results also in the presence of small measurement noise.

\begin{figure}
	\centering
	\includegraphics[width=0.55\linewidth,trim={0 0 0 1.5cm},clip]{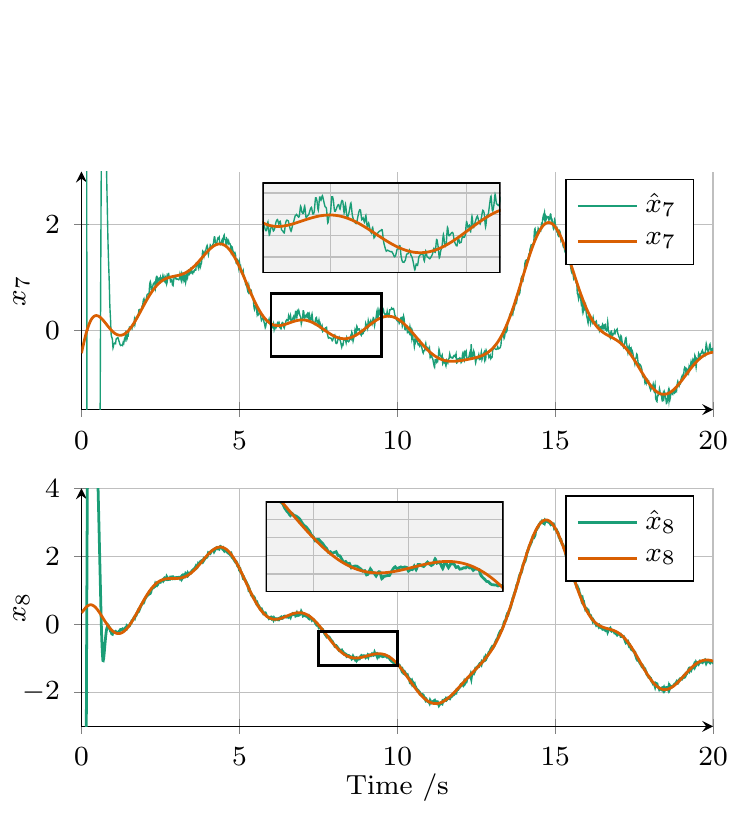}
	\caption{True and reconstructed states $x_7$ and $x_8$ for noisy measurements.}
	\label{fig:noisystates}
\end{figure}

\section{CONCLUSION AND OUTLOOK}\label{sec:conclusion}
This paper discusses stability and robustness properties of a recently proposed observer algorithm for linear time varying systems. The observer is numerically efficient, especially if the non-stable subspace is of low dimension. The resulting error system is analyzed with respect to uniformly bounded disturbance inputs and conditions for bounded input bounded state stability based on Lyapunov exponents are presented. Based on this observer, a finite time exact state reconstruction scheme utilizing higher order sliding mode differentiators is presented for a special class of linear time varying systems. Future topics of interest are the relaxation of the assumed smoothness properties, especially the differentiability assumption on the unknown input and the extension of the concept to nonlinear systems.

\section*{Acknowledgments}\ \\
This work was partially supported by the Graz University of Technology LEAD project ``Dependable Internet of Things in Adverse Environments''.  L. Fridman acknowledges the Financial support of CONACyT (Consejo Nacional de Ciencia y Tecnologa): Project 282013, PAPIIT-UNAM (Programa de Apoyo
a Proyectos de Investigación e Innovación Tecnológica) IN
113216, PASPA DGAPA of UNAM.

\FloatBarrier

\bibliography{Literatur}

\end{document}